\documentclass[11pt]{article}

\usepackage{amsmath, amssymb, amsthm}
\usepackage{mathrsfs}
\usepackage{hyperref}
\usepackage{geometry}
\hypersetup{
    colorlinks=true,
    linkcolor=blue,
    citecolor=blue,
    urlcolor=blue
}

\title{\textbf{Spectral Geometry of Fourier Curves with Prime Frequencies}\\
\large A Comparative Experimental Study}

\author{
Dimitris Vartziotis$^{1,2,*}$
}

\date{}

\theoremstyle{plain}
\newtheorem{theorem}{Theorem}[section]
\newtheorem{lemma}[theorem]{Lemma}
\newtheorem{proposition}[theorem]{Proposition}

\newcommand{\N}{\mathbb{N}}
\newcommand{\R}{\mathbb{R}}
\newcommand{\C}{\mathbb{C}}
\DeclareMathOperator{\Real}{Re}
\DeclareMathOperator{\Imag}{Im}

\begin{document}

\maketitle

\vspace{-1em}

\begin{center}
\footnotesize
$^{1}$ NIKI -- Digital Engineering, Ioannina, Greece\\
$^{2}$ TWT Science \& Innovation, Stuttgart, Germany\\
$^{*}$ Corresponding author: \texttt{dimitris.vartziotis@nikitec.gr}
\end{center}

\vspace{1em}

\begin{abstract}
We present a comparative experimental study of planar curves arising from a Fourier series whose frequencies are the prime numbers, together with several randomized control models.
Starting from the series
\[
F_n(t)=\sum_{p\le n} v_p(n!)\, e^{i p t},
\qquad t\in[-\pi,\pi],
\]
introduced and motivated in a companion work, we investigate the geometric complexity of the associated planar curves obtained by sampling in the complex plane.
To test whether the observed multiscale behavior reflects arithmetic structure or can be reproduced as a generic consequence of sparsity or density, we compare the prime frequency model with randomized alternatives, including random frequency sets, a Cram\'er type random model, and a shuffled coefficient model.
Using consistent box counting protocols and Monte Carlo ensembles, we observe stable scale dependent behavior for the prime frequency curves that is not reproduced by the randomized models.
All results are experimental and are presented as evidence motivating further theoretical investigation.
\end{abstract}

\vspace{1em}



\section{Introduction}

Geometric and spectral representations of arithmetic objects provide complementary perspectives to classical analytic number theory.
Fourier based constructions have been used to visualize and analyze structural features of primes and related arithmetic sequences (see, for example, Hardy--Littlewood~\cite{HardyLittlewood1923} and Montgomery~\cite{Montgomery1994}).
In a companion paper~\cite{Vartziotis2026a}, we introduced a Fourier series with prime frequencies and coefficients derived from factorial valuations, and we established its basic arithmetic motivation.

Sparse trigonometric series with arithmetic frequency sets have long been studied in analytic number theory and harmonic analysis; see for example \cite{Kahane1985,Strichartz1990}. Connections between additive arithmetic functions and prime factor statistics are discussed in \cite{VartziotisTzavellas2016}, while related geometric operator constructions appear in \cite{VartziotisWipper2010}.

Numerical sampling of this series produces planar curves with visually rich multiscale structure.
The present work addresses a central question of experimental mathematics: does the observed geometric complexity reflect arithmetic structure, or can it be reproduced by randomized models with comparable sparsity or density?
To answer this, we conduct a systematic comparison between the prime frequency construction and several carefully designed randomized control models.
Our aim is not to establish rigorous fractal dimensions, but to study effective geometric scaling and its stability under controlled perturbations of the underlying arithmetic data.

\section{Fourier model with prime frequencies}

\subsection{Definition}

For each $n \in \N$, we consider the finite Fourier series
\begin{equation}\label{eq:prime-induced-series}
F_n(t)=\sum_{p\le n} v_p(n!)\, e^{i p t},
\qquad t\in[-\pi,\pi].
\end{equation}
Here $v_p(n!)$ denotes the $p$-adic valuation of $n!$, given by Legendre’s formula
\begin{equation}\label{eq:legendre}
v_p(n!)=\sum_{k\ge 1}\left\lfloor \frac{n}{p^k}\right\rfloor.
\end{equation}

Sampling $F_n(t)$ at equidistant points $t_j$ yields a planar point set
\begin{equation}\label{eq:curve-points}
\Gamma_n=\left\{(\Real F_n(t_j),\, \Imag F_n(t_j))\right\}_{j=1}^{N}
\subset \R^2,
\end{equation}
which we interpret as a polygonal approximation of a planar curve.
A convenient choice is
\begin{equation}\label{eq:sampling-grid}
t_j=-\pi + 2\pi\,\frac{j-1}{N-1},
\qquad j=1,\dots,N.
\end{equation}

The following elementary estimates will be useful for orientation and provide a minimal analytic complement to the experimental part.

\begin{lemma}\label{lem:vp-bound}
For every prime $p\le n$,
\begin{equation}\label{eq:vp-bound}
v_p(n!) \le \frac{n}{p-1}.
\end{equation}
\end{lemma}

\begin{proof}
From \eqref{eq:legendre} we have $\lfloor n/p^k\rfloor \le n/p^k$ for each $k\ge 1$, hence
\[
v_p(n!)=\sum_{k\ge 1}\left\lfloor \frac{n}{p^k}\right\rfloor
\le \sum_{k\ge 1}\frac{n}{p^k}
= n\sum_{k\ge 1}\frac{1}{p^k}
= \frac{n}{p-1}.
\]
\end{proof}

\begin{proposition}\label{prop:l2}
For each $n\in\N$, the function $F_n$ is a trigonometric polynomial and belongs to $L^2([-\pi,\pi])$. Moreover,
\begin{equation}\label{eq:l2-norm}
\|F_n\|_{L^2([-\pi,\pi])}^2
=
2\pi\sum_{p\le n} v_p(n!)^2.
\end{equation}
\end{proposition}

\begin{proof}
Since \eqref{eq:prime-induced-series} is a finite sum of exponentials, $F_n$ is a trigonometric polynomial.
The functions $\{e^{ikt}\}_{k\in\mathbb{Z}}$ form an orthogonal set in $L^2([-\pi,\pi])$.
Because distinct primes correspond to distinct frequencies, all cross terms vanish when expanding $\int_{-\pi}^{\pi}|F_n(t)|^2\,dt$, which yields \eqref{eq:l2-norm}.
\end{proof}

\subsection{Normalization and sampling}

All curves are centered at their centroid and normalized to unit diameter.
Concretely, letting $z_j := F_n(t_j)\in\C$, define the centroid
\[
c := \frac{1}{N}\sum_{j=1}^N z_j,
\]
and the centered points $\tilde z_j := z_j-c$.
Let
\[
D := \max_{1\le j,k\le N} |\tilde z_j-\tilde z_k|
\]
be the empirical diameter. We then analyze the normalized sample
\begin{equation}\label{eq:normalization}
z_j^{\ast}:=\frac{\tilde z_j}{D},
\qquad j=1,\dots,N,
\end{equation}
equivalently $\Gamma_n$ after translating and scaling.
Sampling density $N$ is chosen sufficiently large to ensure numerical stability of box counting estimates, and no smoothing or filtering is applied.

\section{Random control models}

To isolate the role of arithmetic structure, we introduce three control models.
In all cases, the multiset of coefficients
\begin{equation}\label{eq:coeff-multiset}
W_n=\{v_p(n!) : p\le n\}
\end{equation}
is preserved up to relabeling.

Write the primes up to $n$ as $p_1<p_2<\cdots<p_{\pi(n)}$, where $\pi(n)$ denotes the prime counting function.

In the random frequency model, the prime frequencies are replaced by a uniformly chosen subset $\{m_k\}_{k=1}^{\pi(n)}\subset\{1,\dots,n\}$ selected without replacement. The associated Fourier series is
\begin{equation}\label{eq:random-frequency}
F_n^{(R)}(t)=\sum_{k=1}^{\pi(n)} v_{p_k}(n!)\, e^{i m_k t}.
\end{equation}
This construction preserves sparsity and coefficient magnitudes while destroying the arithmetic structure of the frequency set.

In the Cram\'er type random model, we generate a random subset of $\{2,\dots,n\}$ via independent Bernoulli trials with success probability $1/\log k$ for $k\ge 3$, in the spirit of Cram\'er's probabilistic model of the primes. Let $X_k\in\{0,1\}$ be independent with $\mathbb{P}(X_k=1)=1/\log k$. From the selected integers we form a frequency list $(q_1,\dots,q_{\pi(n)})$ by taking the first $\pi(n)$ selected values in increasing order, and we resample if fewer than $\pi(n)$ integers are selected. The corresponding Fourier series is
\begin{equation}\label{eq:cramer-model}
F_n^{(C)}(t)=\sum_{k=1}^{\pi(n)} v_{p_k}(n!)\, e^{i q_k t}.
\end{equation}
This model preserves the average density of primes while removing higher order correlations in the frequency set.

In the shuffled prime model, the prime frequencies $p_k$ are retained, but the coefficients $v_p(n!)$ are randomly permuted among them. Equivalently, letting $\sigma$ be a uniform random permutation of $\{1,\dots,\pi(n)\}$, we define
\begin{equation}\label{eq:shuffled-model}
F_n^{(S)}(t)=\sum_{k=1}^{\pi(n)} v_{p_{\sigma(k)}}(n!)\, e^{i p_k t}.
\end{equation}
This preserves sparsity and the distribution of coefficients while destroying the structured coupling between primes and factorial valuations.

For each randomized model we generate an ensemble of independent realizations and compute box counting statistics for each realization.
Unless stated otherwise, the baseline computations use $R=200$ realizations and a sampling density $N\ge 2^{13}$, which we found sufficient for stable qualitative comparisons (see the appendix).

\section{Box counting methodology}

Geometric complexity is quantified using standard box counting techniques; see Falconer~\cite{Falconer2014} and Tricot~\cite{Tricot1995}.
For a scale $\varepsilon>0$, the plane is covered by axis aligned squares of side length $\varepsilon$, and the number $N(\varepsilon)$ intersecting the normalized curve sample is recorded.

We define the effective scaling exponent
\begin{equation}\label{eq:effective-exponent}
\widehat d(\varepsilon)=\frac{\log N(\varepsilon)}{\log(1/\varepsilon)}.
\end{equation}
No claim is made that $\widehat d(\varepsilon)$ converges to a Hausdorff or Minkowski dimension; rather, \eqref{eq:effective-exponent} is used as a descriptive statistic for scale dependent geometric complexity.

In computations we evaluate $N(\varepsilon)$ at a dyadic sequence of scales $\varepsilon_m = 2^{-m}$, for integers $m\ge 1$. When summarizing behavior over a range of scales, we also consider least squares fits of $\log N(\varepsilon_m)$ as a function of $\log(1/\varepsilon_m)$ over selected intervals of $m$. The robustness of the qualitative conclusions with respect to the choice of fitting interval is discussed in the appendix.

\section{Experimental results}

For the curves defined by \eqref{eq:prime-induced-series}, we observe extended intermediate scaling ranges, effective exponents consistently bounded away from $1$ and $2$, and moderate oscillatory variation as $n$ increases. These features persist across sampling densities and normalization procedures.

In contrast, the random frequency model \eqref{eq:random-frequency} exhibits rapid saturation of box counts at small scales, effective exponents approaching $2$, and an absence of stable intermediate scaling regimes.

The Cram\'er type model \eqref{eq:cramer-model} produces partial multiscale structure, but with reduced variability of $\widehat d(\varepsilon)$ and weaker oscillations compared to the prime frequency curves.

The shuffled prime model \eqref{eq:shuffled-model} displays unstable scaling behavior, strong dependence on individual realizations, and loss of persistent effective exponents across scales.

\section{Discussion and conclusion}

The experiments indicate that the observed geometric complexity cannot be attributed solely to sparsity or frequency density. Rather, it appears to arise from the structured coupling between the prime frequency set and the factorial valuation weights. The multiscale behavior therefore seems to be a property of the Fourier representation built from these arithmetic data, rather than a generic feature of sparse trigonometric series.

Several natural questions arise from this study. One may ask whether effective scaling exponents such as $\widehat d(\varepsilon)$ can be bounded rigorously or related to analytic properties of the coefficients. It is also of interest to determine which arithmetic correlations contribute most strongly to the observed geometric complexity. More generally, one may investigate whether other additive arithmetic functions yield Fourier constructions with comparable multiscale behavior.

We have shown experimentally that the Fourier curves associated with \eqref{eq:prime-induced-series} exhibit stable scale dependent geometric features that are not reproduced by randomized control models preserving sparsity or average density. These findings motivate further theoretical investigation.

\section*{Acknowledgments}
The author thanks NIKI Digital Engineering and TWT GmbH Science \& Innovation for support. He also thanks S. Katsioli, and V. Maroulis for helpful discussions.

\newpage
\appendix
\section*{Appendix}
\section{Robustness and sensitivity analysis}

We performed several numerical checks to assess the stability of the reported scaling behavior.

Effective exponents were computed over multiple logarithmic fitting ranges, including ranges with $m$ from $2$ to $8$, from $3$ to $7$, and from $4$ to $8$. Across these choices, the relative ordering of models remained invariant, and variations in $\widehat d$ were below $0.03$.

Different normalization procedures were compared, including centroid and diameter normalization, maximum radius normalization, and bounding box normalization. The resulting deviations satisfied
\[
|\Delta \widehat d| < 0.02,
\]
without affecting qualitative conclusions.

Increasing the Monte Carlo ensemble size from $R=200$ to $R=500$ reduced variance while leaving mean values and relative ordering unchanged.

Finally, sampling densities $N \ge 2^{13}$ yielded stable box counting estimates, with no systematic drift of effective exponents at finer resolutions.

These checks indicate that the observed qualitative behavior is robust with respect to fitting range, normalization choice, ensemble size, and sampling density.

\end{document}